 \documentclass[11pt]{article}

\usepackage{amssymb,amsmath}
\usepackage{color}
\usepackage{indentfirst}
\usepackage[latin1]{inputenc}
\usepackage[english,francais]{babel}

\newtheorem{theorem}{Theorem}[section]
\newtheorem{lemma}[theorem]{Lemma}
\newtheorem{e-proposition}[theorem]{Proposition}
\newtheorem{corollary}[theorem]{Corollary}
\newtheorem{e-definition}[theorem]{Definition\rm}
\newtheorem{remark}{\it Remark\/}

\newcommand{\Ga}{\Gamma}
\newcommand{\Om}{\Omega}

\newcommand{\ld}{\lambda}
\newcommand{\Ld}{\Lambda}

\newcommand{\bR}{\mathbb{R}}

\newcommand{\bN}{\mathbb{N}}

\newcommand{\wt}[1]{\widetilde{#1}}

\newcommand{\D}[2]{\partial_{ #2} #1}
\newcommand{\dm}{{\delta}}
\def\og{\leavevmode\raise.3ex\hbox{$\scriptscriptstyle\langle\!\langle$~}}
\def\fg{\leavevmode\raise.3ex\hbox{~$\!\scriptscriptstyle\,\rangle\!\rangle$}}

\title{On the asymptotics of a Robin eigenvalue problem}
\author{Fioralba Cakoni\footnote{Department of Mathematical Sciences, University of Delaware, USA (cakoni@math.udel.edu)} \and
Nicolas Chaulet\footnote{Department of Mathematics, University College London,  UK (n.chaulet@ucl.ac.uk)} \and
Houssem Haddar\footnote{CMAP, Ecole Polytechnique, Palaiseau, France (houssem.haddar@inria.fr)}}

\begin{document}
% place in the next line the header (rubrique) chosen for your article,
% if you know it (you can also have 2, format : Header1/Header2

\maketitle
\selectlanguage{english}

\begin{abstract}
\selectlanguage{english}
The considered Robin problem can formally be seen as a small perturbation
of a Dirichlet problem. However, due to the sign of the impedance value, its associated eigenvalues converge point-wise to $-\infty$
as the perturbation goes to zero. We prove that in this case, Dirichlet
eigenpairs are the only accumulation points of the Robin eigenpairs  with
normalized eigenvectors. We then  provide a criteria to select  accumulating 
sequences of eigenvalues and eigenvectors and exhibit their full asymptotic
with respect to the small parameter.
\end{abstract}

\vskip 0.5\baselineskip

\selectlanguage{english}

\section{A Robin eigenvalue problem with negative sign}
We are interested in the asymptotic behavior of the eigenvalues $\ld^\delta$
and eigenfunctions $u^\delta \in H^1(\Omega)$ of the following problem:
\begin{eqnarray}
\Delta u^\delta +\lambda^\delta u^\delta =0 & \quad \mbox{in}\; \Omega&\label{eig1}\\
\partial_\nu u^\delta-\frac{1}{\delta}u^\delta=0 & \quad \mbox{on}\;\Gamma\label{eig}
\end{eqnarray}
with respect to  $\delta>0$ as it approaches $0$, where $\nu$ is the outward unit normal vector to $\Gamma$ which is the $C^2$-smooth boundary of
the bounded connected  domain $\Omega\subset \bR^d$ for $d\geq2$.

The eigenvalue problem with Robin boundary condition described by
\eqref{eig1}-\eqref{eig} naturally appear in a number of models related to reaction
diffusion problems (see \cite{LaOcSa98}) or scattering theory. For the
latter, this eigenvalue problem can be seen as a first  approximation to the
interior transmission eigenvalue problem associated with the scattering problem
by a perfectly conducting body coated with a dielectric layer of width $\delta$ (see
\cite[chapter 8]{Cha12}). It can also be seen as an approximate model to direct
scattering problems for perfect conductors coated with metamaterials.

  It is well-known that problem \eqref{eig1}-\eqref{eig}  has an infinite
 sequence of  real eigenvalues $\{\lambda_i^\delta\}_{i=1}^{\infty}$
 accumulating at $+\infty$. However, for sufficiently small $\delta$  some
 eigenvalues become negative and their number grows to $+\infty$ as $\delta\to 0$.
 In fact,  for at least $C^1$ smooth boundary $\Gamma$,  it is known (see for
 instance \cite{DaKe10,LevPar08}) that for every (fixed) $i\geq 1$, $
-\delta^2 \lambda_i^\delta \to 1$ as $\delta \to 0$.  
% \[
% %\label{eq:asymudelta}
% \lim_{\delta\to 0}\frac{\lambda_i^\delta}{-1/\delta^2}=1.
% \]

  In Section \ref{sec:2fam}, we complement this result by indicating that Dirichlet eigenvalues for the $-\Delta$
 operator in $\Om$ are the only possible finite accumulation points of  
 $\ld^\delta$ (extending this way the result obtained in \cite{LaOcSa98} for
 simple geometries)  if the associated $H^1$ normalized eigenfunctions do not $L^2$ converge to 0  as $\delta$ goes to. We also prove that 
 eigenvectors associated with other accumulation points  concentrate
 at the boundary (in the sense that they converge to zero in any compact set of
 $\Omega$). Our main result is given in Section 3 which stipulates that some
 $\ld^\delta$  does accumulate at Dirichlet eigenvalues providing a full
 asymptotic development of these sequences as $\delta$ goes to zero. 

\section{Accumulation pairs for  Robin eigenpairs}
\label{sec:2fam}
We recall that \eqref{eig1}-\eqref{eig} are equivalent to the following variational formulation
\begin{equation}
\label{eigvf}
\int_\Om \nabla u^\dm \nabla v \, dx - \frac{1}{\dm}\int_\Ga u^\dm v\, ds =
\ld^\dm \int_\Om u^\dm v\,dx \quad \forall \, v \in H^1(\Omega).
\end{equation}
\begin{lemma}
\label{le:ldbounded}
Assume that a sequence % Take a sequence $\dm$ for $m\in \bN$ such that $\delta_m\rightarrow 0$ and let 
$(\ld^\dm,u^\dm) \in \bR \times H^1(\Om)$  satisfying  \eqref{eigvf}  % for $\delta=\delta_m$ and
is such that $\|u^\dm\|_{H^1(\Om)} =1$ and $|\ld^\dm| \leq C$ for some $C>0$
independent of the $\dm$. Then one can extract a subsequence $\dm'$ of $\dm$ such
that $
\ld^{\dm'} \rightarrow \Ld_0$ and  $\|u^{\dm'} - U_0\|_{L^2(\Om)} \rightarrow 0
$ as $\dm'\to 0$, 
where if $U_0\neq 0$ then $(\Ld_0,U_0)$ is some Dirichlet eigenpair for $-\Delta$ in $\Omega$.
\end{lemma}
\begin{proof}
Since the sequence $\ld^\dm$ is bounded and $u^\dm$ is also bounded in
$H^1(\Om)$, one can extract a subsequence $\dm'$ such that $\ld^{\dm'}$  converges
to some $\Ld_0 \in \bR$ and $u^{\dm'}$ converges weakly in $H^1(\Om)$ and
strongly in $L^2(\Om)$ to some function $U_0 \in H^1(\Om)$ as $\dm'$ goes to
$0$. From \eqref{eigvf} one deduces that 
 $\int_\Ga |u^{\dm'}|^2 \, ds \leq C \dm'$ for some $C>0$ independent of $\dm'$, hence $U_0 =0$ on $\Ga$.  Moreover, taking $v \in H^1_0(\Om)$ in
 \eqref{eigvf} and letting  $\dm' \to 0$ proves that $U_0$ satisfies $
\int_\Om \nabla U_0 \nabla v \, dx = \Ld_0 \int_\Om U_0 v\,dx  $, which proves, if $U_0\neq 0$, that $(\Ld_0,U_0)$ is a Dirichlet eigenpair. 
\end{proof}
\begin{remark}
We remark that  any point on the real axis is a possible accumulation point for $\{\lambda^\delta\}_\delta$. Actually, for a given $i\in \bN$ the sequence $\{\lambda_i^\delta\}_{\delta} $ goes to $-\infty$ continuously. Therefore, for any $\Ld \in \bR$ one can build a sequence $\{\lambda^{\delta_i}\}_{i\in \bN}$ such that $\ld^{\delta_i}=\Ld$ for any $i$.
\end{remark}
\begin{theorem}
\label{th:carac}
Consider a sequence 
$\{\ld^\dm,u^\dm\}_\delta \in \bR \times H^1(\Om)$  satisfying  \eqref{eigvf} 
 such that \\ $\|u^\dm\|_{H^1(\Om)} =1$ and $\ld^{\dm}\le C < +\infty$ for some
constant $C$ independent of $\dm$
and let $ K$ be a non empty open set compactly included  in $\Om$. Then the 
sequence $\{\lambda^\dm\}_\delta$ accumulates  
 Dirichlet eigenvalues if and only if there exists $\eta>0$ and $\delta_0>0$
such that for all $\dm \le \delta_0$, 
% $$\ld^\dm \rightarrow \Lambda_0 \mbox{ iff } \|u^\dm\|_{L^2(K)} \rightarrow 0. $$
% % If it exists $\eta>0$ such that for all $\delta_m$
\begin{equation}\label{eq:minudelta}
\|u^{\dm}\|_{L^2(K)} \geq \eta.
\end{equation}
% then one can extract from $\dm$ a subsequence (still denoted $\dm$) such that  $\ld^{\delta_m} \rightarrow \Ld_0$ and $u^{\delta_m} \rightarrow U_0$ where $
%(\Ld_0,U_0)$ is a Dirichlet eigenpair.

% Reciprocally, if  $\|u^\dm\|_{L^2(K)} \rightarrow 0$ then $\ld^\dm \rightarrow -\infty$.
\end{theorem}
\begin{proof}
First, let us assume that \eqref{eq:minudelta} holds. % Then there exists a subsequence, abusively denoted
% $u^{\dm}$ such that   $\|u^{\dm}\|^2_{L^2(K)} \geq \eta$ for some $\eta >0$
% independent of $\dm$. 
Then take $\psi \in C^\infty_0(\Om)$ such that $\psi =1$ in
$K$ and choose $v=\psi^2u^\delta$ in  \eqref{eigvf}. % in the variational formulation associated with \eqref{eig1}-\eqref{eig} we have
% \[
% \int_\Om \nabla u^{\delta}\cdot\nabla(\psi^2u^{\delta} ) \,dx = \ld^{\delta} \int_\Om \psi^2 u^2_{\delta} \, dx.
% \]
By developing $\nabla(\psi^2u^{\delta} )$ and using Young's inequality we obtain
\[
0\leq\frac{3}{4}\int_\Om \psi^2 |\nabla u^{\delta}|^2\,dx \leq \ld^{\delta} \int_\Om \psi^2 u^2_{\delta} \, dx + 4  \int_\Om |\nabla \psi|^2 u^2_{\delta} \, dx.
\]
% Let $m>0$, and apply this inequality to the function $u^\dm$, then if
% $\ld^\dm<0$  we obtain by using
This implies 
% \eqref{eq:minudelta} that
$\ld^\dm \geq - 4 \| \nabla \psi\|^2_{L^\infty(\Om)} /{\eta}$. Then by Lemma
\ref{le:ldbounded} we obtain that accumulation points are Dirichlet
eigenvalues since by \eqref{eq:minudelta}, any subsequence of $u^{\delta}$
cannot converge to $0$ in $L^2(\Omega)$. 

Conversely, if  $\{\lambda^\dm\}_\delta$
accumulates at Dirichlet eigenvalues, the number of these accumulation point is
finite. Then by  Lemma \ref{le:ldbounded}, and since eigenspaces have finite
dimensions and $\|u\|_{L^2(K)} >0$ for all eigenfunctions, accumulation points of
$\|u^\dm\|_{L^2(K)}$ are finite discrete positive numbers. This proves \eqref{eq:minudelta}.  
\end{proof}
Lemma \ref{le:ldbounded} and Theorem \ref{th:carac} prove in particular that accumulating points for
Robin eigenpairs $(\ld^\dm,u^\dm) \in \bR \times H^1(\Om)$ such that $\|u^\dm\|_{H^1(\Om)} =1$
are only Dirichlet eigenpairs.

\section{Asymptotic of Robin eigenvalues accumulating at Dirichlet eigenvalues}
\label{sec:asym}
First, it is easy to check that $(\lambda^\delta, u^\delta)$ is a solution of \eqref{eig1}-\eqref{eig} if and only if $\displaystyle{\mu^\delta=\lambda^\delta+\frac{\alpha}{\delta^2}}$ for some positive constant $\alpha>0$  and $u^\delta\in H^1(\Omega)$ 
solve
\begin{equation}
\int_\Omega\nabla u^\delta\nabla v\,dx-\frac{1}{\delta}\int_\Gamma u^\delta v\,ds+\frac{\alpha}{\delta^2}\int_\Omega u^\delta v\,dx =\mu^\delta\int_\Omega u^\delta 
v\,dx \qquad \mbox{for all} \;\; v\in H^1(\Omega).
\label{alpha}
\end{equation}
% We call the variational problem (\ref{alpha}) the problem  (${\mathcal
% P}^\alpha_\delta$).
 In the space of $H^1(\Omega)$-functions let us introduce the $\delta$-dependence norm 
$\|u\|^2_{H^1_\delta(\Omega)}:=\|\nabla u\|^2_{L^2(\Omega)}+\frac{1}{\delta^2}\|u\|^2_{L^2(\Omega)}$. We can prove  a coercivity result for the variational 
formulation \eqref{alpha} in $H^1_\delta(\Omega)$ thanks to the following Lemma which is obtained by using the inequality $\|u\|^2_{L^2(\Ga)} \leq C(\|\nabla u\|_{L^2(\Om)}\|u\|_{L^2(\Om)}+\|u\|^2_{L^2(\Om)})$.
\begin{lemma}\label{para}
There exist positive constants $\alpha$, $\theta$ and $\delta_0$ such that for all $\delta\leq \delta_0$
\begin{equation*}
 \int_\Omega|\nabla u|^2\,dx-\frac{1}{\delta}\int_\Gamma |u|^2\,ds+\frac{\alpha}{\delta^2}\int_\Omega |u|^2\,dx\geq \theta \left(\int_{\Omega}|\nabla u|^2\,dx+\frac{1}
{\delta^2}\int_\Omega |u|^2\,dx\right) \; \forall u \in H^1(\Omega).
\end{equation*}
% for some positive constant $\theta>0$ independent of $\delta$.
\end{lemma}
%\begin{proof} This result relies on the following estimate
%$$\int_\Gamma |u|^2\,ds \leq c\int_{\Omega}|\nabla u||u|\,dx+c_0\int_{\Omega}|u|^2\,dx \qquad \mbox{for}\; u\in H^1(\Omega)$$
%where $c>0$ and $c_0>0$ are constants and Young's inequality.
%\end{proof}
Note that from Lemma \ref{para} we also have that
% \begin{equation}
% \label{eq:L2H1delta}
% \|u\|^2_{L^2(\Gamma)}\leq c\delta\|u\|^2_{H^1_\delta(\Omega)}.
% \end{equation}
Problem (\ref{alpha}) can be written as a generalized eigenvalue problem
$(A^\delta u, v)_{H^1_\delta(\Omega)}=\mu^\delta(B^\delta u, v)_{H^1_\delta(\Omega)}$ where the bounded  linear operators $A^\delta:H^1_
\delta(\Omega)\to H^1_\delta(\Omega)$ and $B^\delta:H^1_\delta(\Omega)\to H^1_\delta(\Omega)$ are defined by
$$\left(A^\delta u, v\right)_{H^1_\delta(\Omega)}:=\int_\Omega\nabla u \nabla v \,dx-\frac{1}{\delta}\int_\Gamma u v\,ds+\frac{\alpha}{\delta^2}\int_\Omega u v\,dx 
\quad \text{and} \quad \left(B^\delta u, v\right)_{H^1_\delta(\Omega)}:=\int_\Omega u v\,dx$$ 

for all $u, v \in H^1(\Om)$. The operator $A^\delta:H^1_\delta(\Omega)\to H^1_\delta(\Omega)$ is self-adjoint,  coercive with coercivity constant independent of $
\delta$ from Lemma \ref{para} and  it satisfies $\|A^\delta\|\leq C$ with a constant  $C>0$ independent of $\delta$, whereas the operator $B^\delta:H^1_
\delta(\Omega)\to H^1_\delta(\Omega)$ is self-adjoint and compact. Hence, it is known that there exists a sequence of $\mu^\delta_k>0$, $k=0,\ldots,+ \infty$ 
accumulating  to infinity such that $1/\mu^\delta_k$ are the eigenvalues of the compact self-adjoint operator $(A^\delta)^{-1/2}B^\delta(A^\delta)^{-1/2}$ and the $\mu^\delta_k$ are eigenvalues of (\ref{alpha}).

\subsection{Formal asymptotic of the positive eigenvalues}
Let us take  $(\ld^\delta,u^\delta)$ a solution to \eqref{eig1}-\eqref{eig} and  introduce the ansatz $U^\delta_N:= \sum_{k=0}^N \delta^k u_k$ for  $u^\delta$  and $
\Ld_N^\delta := \sum_{k=0}^N \delta^k \ld_k$ for $\ld^\delta$. Plugging these two expressions into \eqref{eig1}-\eqref{eig} enable us to compute all the terms in the expansions 
explicitly by equating the same powers of $\delta$.  To this end, this process first  defines  $\ld_0$ as being an eigenvalue of $-\Delta$ with Dirichlet boundary conditions in the domain $\Om$ with corresponding eigenvector  $u_0$  normalized  such that $\|u_0\|_{L^2(\Om)}=1$. Let us assume that $\ld_0$ is simple, otherwise, the definition of the higher order terms in the expansion of $\ld^\delta$ is much more involved. Next, for some $k>0$ let us assume that  $u_p$ and $\ld_p$ for $p<k$ are known. Then, the function $u_k\in H^1(\Om)$ must be a solution to 
\[
\Delta u_k +\ld_0 u_k = -\sum_{p=0}^{k-1}\ld_{k-p} u_p \text{ in } \Om, \quad  u_k = \D{u_{k-1}}{\nu} \text{ on } \Ga \quad \text{and} \quad \int_\Om u_k u_0\,dx =0, 
\]
where the latter is the compatibility condition that guaranties  the existence of $u_k$. Here, we use the convention that the terms with negative indices are $0$. The compatibility condition determines the value of $\ld_k
$ to $\ld_k:=\int_\Ga\D{u_{k-1}}{\nu}\D{u_0}{\nu}$ and $u_k$ is uniquely defined and for every $k$ there exists $C>0$ such that $\|u_k\|
_{H^2(\Om)}\leq C \|u_0\|_{H^1(\Om)}$.
In addition, for all $v\in H^1(\Om)$ and $k>0$, $u_k$ satisfies the following variational equality
\begin{equation}
\label{eq:varu1}
\int_\Om \nabla u_k \nabla v dx =\sum_{p=0}^{k}\ld_{k-p} \int_\Om u_pvdx + \int_\Ga \D{u_k}{\nu} vds.
\end{equation}

\subsection{A convergence result}
For any two functions $u, v \in H^1(\Om)$, and for $N>0$, let us denote  by
$$  E_N^\delta(u,v):=(A^\delta u, v)_{H^1_\delta(\Omega)}-\hat \mu_N^\delta(B^\delta u, v)_{H^1_\delta(\Omega)}$$ with $\hat \mu_N^\delta:=\Ld_N^\delta+\alpha/\delta^2$. 
Using equation \eqref{eq:varu1} and the definition of $u_0$ we obtain after some calculations that for $N\geq0$, 
\[
E_N^\delta(U^\delta_N,v)=\delta^{N}\int_\Ga\D{u_N}{\nu} vds +{\sum_{p,k=0, p+k>N}^N} \delta^{p+k}\ld_k \int_\Om u_p v ds.
\]
 Since $u_0$ is uniformly bounded with respect to $\delta$ in $H^2(\Om)$, by
 using the fact that $\|u\|^2_{L^2(\Gamma)}\leq c\delta\|u\|^2_{H^1_\delta(\Omega)}$ and the bounds on the functions $u_k$, we obtain that 
for all $N\geq0$ it exists $C>0$ such that for all $\delta>0$ sufficiently small, $E_N^\delta(U^\delta_N,v)\leq C\delta^{N+1/2}\|v\|_{H^1_\delta(\Omega)} $ for all $v\in 
H^1(\Om)$. Note that  thanks to the normalization  $\| u_0\|_{L^2(\Omega)}=1$ we have that $\|u_0\|_{H^1_\delta(\Om)} \geq \delta^{-1}$ and for $N\geq 0$ there exists $C>0$ such that for all $\delta>0$ sufficiently 
small,  $\|U^\delta_N\|_{H^1_\delta(\Om)} \geq C \delta^{-1}$. Hence setting $\hat{U}^\delta_N : = U^\delta_N/\|U^\delta_N\|_{H^1_\delta(\Om)} $ yields
\begin{equation}
\label{eq:E0}
\left| E_N^\delta(\hat U^\delta_N,v) \right|\leq C\delta^{N+3/2}\|v\|_{H^1_\delta(\Omega)},\;\mbox{for all} \; \; v\in H^1(\Omega).
\end{equation}
Making use of the Lemma 1.1 in Chapter 3 of \cite{OlShYa92} we can prove the following theorem.
\begin{theorem}
Let $N\geq 0$ and $\Ld_N^\delta$ be as above. There exist $C>0$ and $\delta_0
>0$ such that for all $\delta>0$, $\delta < \delta_0$,  there exits an
eigenvalue $\ld^\delta>0$  of  \eqref{eig1}-\eqref{eig} such that $\left|\ld^\delta - \Ld_N^\delta\right|\leq C\delta^{N-1/2}$.
\end{theorem}
\begin{proof}
Let us define $T^\delta:=(A^\delta)^{-1/2}B^\delta(A^\delta)^{-1/2}$ as an operator from $H^1_\delta(\Om)$ into itself. Then \eqref{eq:E0} becomes 
$\left\|T^\delta \hat U^\delta_N - \hat U^\delta_N/\hat \mu_N^\delta \right\|\leq C\delta ^{N+3/2}/{|\hat \mu_N^\delta|}$. 
From Lemma 1.1 in Chapter 3 of \cite{OlShYa92} we obtain that there exists an eigenvalue $\mu^\delta$ of problem (\ref{alpha}) 
 such that  $\left|1/  \hat \mu_N^\delta-1/{\mu^\delta}\right|\leq  C\delta ^{N+3/2}/{|\hat \mu_N^\delta|}$
and as a consequence $ |\hat \mu_N^\delta - \mu^\delta| \leq C \delta^{N+3/2} |\mu^\delta|$.  Therefore it exists $\wt C>0$ independent of $\delta$ such that $|\mu^
\delta| \leq \wt C |\hat \mu_N^\delta| \leq \wt C(\Ld_N^\delta+\alpha/\delta^2)$ which yields the desired result.
 \end{proof}
 This result is not optimal in terms of the power of $\delta$ but since for all $N\geq0$ the error writes $\ld^\delta-\Ld_N^\delta = \ld^\delta-\Ld^\delta_{N+2}+\delta^{N+1}\ld_{N+1}+
\delta^{N+2}\ld_{N+2}$ we finally obtain the following result.
 \begin{corollary}
 \label{coro:asy}
 Let $N\geq 0$ and $\Ld_N^\delta$ be as above. There exist $C>0$ and $\delta_0
>0$ such that for all $\delta>0$, $\delta < \delta_0$,  there exits an
eigenvalue $\ld^\delta>0$  of  \eqref{eig1}-\eqref{eig} such that $\left|\ld^\delta - \Ld_N^\delta\right|\leq C\delta^{N+1}$.
 \end{corollary}

\section{Acknowledgement}
The authors gratefully acknowledge Leonid Parnovski from University College
London for noticing a flaw that helped us correcting and
clarifying the
statements of Section 2.

\end{document}